\pdfoutput=1
\documentclass[11pt]{article}
\usepackage[T1]{fontenc}
\usepackage{lmodern}
\usepackage[margin=1in]{geometry}
\usepackage{amsmath,amssymb,amsthm,mathtools}
\usepackage{tikz}
\usetikzlibrary{arrows.meta}
\usepackage{microtype,caption}
\usepackage{xurl}
\usepackage[hidelinks]{hyperref}
\hypersetup{
  pdftitle={A Coarse-Geometric Interpretation of Holt's One-End Theorem},
  pdfauthor={Hussain Al-Rasheed},
  pdfkeywords={ends of groups, locally finite groups, almost invariant sets, coarse geometry, Higson corona}
}
\newtheorem{theorem}{Theorem}[section]
\newtheorem{lemma}[theorem]{Lemma}
\newtheorem{proposition}[theorem]{Proposition}
\newtheorem{corollary}[theorem]{Corollary}
\theoremstyle{definition}

\newtheorem{example}[theorem]{Example}
\theoremstyle{remark}

\newcommand{\FundingDisclosure}{No funding was received for this work.}
\newcommand{\InterestDisclosure}{The author has no relevant financial or non-financial interests to disclose.}

\title{A Coarse-Geometric Interpretation of Holt's One-End Theorem}
\author{Hussain Al-Rasheed\\[0.5em]
  \small General Studies Department, Jubail Industrial College\\
  \small Royal Commission for Jubail and Yanbu\\
  \small Jubail Industrial City 31961, Saudi Arabia\\
  \small Corresponding author: \href{mailto:rashedhs@rcjy.edu.sa}{\nolinkurl{rashedhs@rcjy.edu.sa}}\\
  \small ORCID: \href{https://orcid.org/0009-0008-1970-4493}{\nolinkurl{0009-0008-1970-4493}}}
\date{}

\begin{document}
\maketitle

\begin{abstract}
Can a group be colored red and blue, with both color classes infinite, so that each right translation changes the color of only finitely many elements? For the group of finite subsets of the natural numbers, one can color a nonempty set by the parity of its largest element (and the empty set red). Holt proved that no such coloring exists in an uncountable locally finite group. For the finitary coarse structure, right almost invariant subsets are exactly the coarsely clopen sets; modulo finite changes, their characteristic functions give all the projections in the Higson corona algebra. We give a detailed, self-contained account of Holt's coset argument in Cornulier's formulation, without using group cohomology. The argument shows why the space of coarse ends is a singleton and the full Higson corona is connected. We then give a proof of Protasov's stronger conclusion for any uncountable abelian group---in particular, an uncountable direct sum of groups of order two---showing that the full Higson corona is a singleton.
\end{abstract}

\par

\section{Introduction}\label{sec:intro}

Identify $G_I \coloneqq \bigoplus_{i\in I}\mathbb{Z}/2\mathbb{Z}$ with the group of finite subsets of $I$, where addition is the symmetric difference $A\triangle B \coloneqq (A\setminus B)\cup(B\setminus A)$.
Can we color the elements of $G_I$ red and blue so that both
color classes are infinite and, for every fixed finite subset
$T\subseteq I$ and for all but finitely many $F\in G_I$,
the elements $F$ and $F\triangle T$ have the same color?
The finite exceptional set may depend on $T$.

When $I=\mathbb{N}=\{1,2,\ldots\}$, let $m(F) \coloneqq \max F$ for $F\neq\varnothing$ and $m(\varnothing) \coloneqq 0$. Color $F$ red if $m(F)$ is even and blue otherwise. For each $n\ge1$, the singleton $\{2n\}$ is red and the
singleton $\{2n-1\}$ is blue, so both color classes are infinite.

Now fix a finite subset $T\subseteq\mathbb{N}$. If $T=\varnothing$, no colors change. Otherwise $m(T)\ge 1$, and whenever
$m(F)>m(T)$, taking the symmetric difference with $T$
neither removes the largest element of $F$ nor introduces
a larger element. Hence
\[
  m(F\triangle T)=m(F),
\]
so the color of $F$ is unchanged. Thus a color change can
occur only when $m(F)\le m(T)$, or equivalently, when
\[
  F\subseteq\{1,\ldots,m(T)\}.
\]
There are only $2^{m(T)}$ such subsets, so each fixed
translation changes the color of only finitely many
elements of $G_{\mathbb{N}}$. This is a concrete instance of the finite-subgroup exhaustion used by Scott and Sonneborn \cite[Theorem~4]{ScottSonneborn}; see also Cornulier \cite[Proposition~7.2]{Cornulier}.
Figure~\ref{fig:translation-colors} illustrates translation by
$\{1\}$, which adds or removes the element $1$.

\begin{figure}[htbp]
\centering
\begin{tikzpicture}[
  font=\small,
  setnode/.style={
    draw, rounded corners=2pt,
    minimum width=23mm, minimum height=8mm,
    inner sep=4pt, text=black
  },
  redset/.style={
    setnode, draw=red!65!black, fill=red!7
  },
  blueset/.style={
    setnode, draw=blue!65!black, fill=blue!7
  },
  translation/.style={<->, >=Latex, semithick},
  annotation/.style={anchor=west}
]
  \node at (0,0.85) {$F$};
  \node at (3.5,0.85) {$F\triangle\{1\}$};

  \node[redset] (empty) at (0,0)
    {$\varnothing$\quad\textsf{R}};
  \node[blueset] (one) at (3.5,0)
    {$\{1\}$\quad\textsf{B}};
  \draw[translation] (empty.east) -- (one.west);
  \node[annotation] at (5.0,0) {color changes};

  \node[redset] (two) at (0,-1.1)
    {$\{2\}$\quad\textsf{R}};
  \node[redset] (onetwo) at (3.5,-1.1)
    {$\{1,2\}$\quad\textsf{R}};
  \draw[translation] (two.east) -- (onetwo.west);
  \node[annotation] at (5.0,-1.1) {color preserved};

  \node[blueset] (three) at (0,-2.2)
    {$\{3\}$\quad\textsf{B}};
  \node[blueset] (onethree) at (3.5,-2.2)
    {$\{1,3\}$\quad\textsf{B}};
  \draw[translation] (three.east) -- (onethree.west);
  \node[annotation] at (5.0,-2.2) {color preserved};
\end{tikzpicture}
\caption{Translation by $\{1\}$ in $G_{\mathbb{N}}$.
The labels R and B denote red and blue.
Only $\varnothing$ and $\{1\}$ change color; for every other
finite subset $F$, the maximum is greater than $1$ and is
unchanged by the translation.}
\label{fig:translation-colors}
\end{figure}

To express the coloring condition algebraically, let
$A\subseteq G_I$ be the red color class. Writing right translation multiplicatively as
\[
  AT \coloneqq \{F\triangle T:F\in A\}
\]
for $T\in G_I$ to match the general group notation below, the symmetric
difference $A\triangle AT$ consists exactly of those elements
whose color changes under this translation (since translation by $T$ is its own inverse). Thus the coloring
condition says that $A\triangle AT$ is finite for every
$T\in G_I$, while the requirement that both colors occur
infinitely often says that $A$ is neither finite nor cofinite.

For a general group $G$, a subset $A\subseteq G$ is called
\emph{right almost invariant} if $Ag\triangle A$ is finite
for every $g\in G$. Holt \cite{Holt1981} proved that every
right almost invariant subset of an uncountable
\emph{locally finite} group is finite or cofinite.
Here, a group is locally finite if every finitely generated
subgroup is finite.

The group $G_I$ is locally finite: finitely many finite
subsets of $I$ generate a subgroup contained in the power
set of their finite union. When $I$ is uncountable, $G_I$
is also uncountable, so Holt's theorem applies.
Consequently, any red--blue coloring of $G_I$ for which each fixed
translation changes the color of only finitely many elements must have
a finite color class. Hence the coloring required in our
opening question exists for $I=\mathbb{N}$ but cannot
exist when $I$ is uncountable. For a finitely generated group, the ends can be seen by removing successively larger finite sets from a Cayley graph \cite{Geoghegan}. A group that is not finitely generated admits no connected, locally finite Cayley graph, yet almost invariant sets still detect ends in this setting.

There are two spaces at infinity to keep apart. The \emph{binary corona}, or space of ends, is determined by slowly oscillating functions (defined in Section~\ref{sec:higson_corona}) with values in $\{0,1\}$, whereas the \emph{full Higson corona} is determined by all bounded complex-valued slowly oscillating functions. Protasov developed the binary-corona construction \cite{Protasov} and explicitly deduced from Holt's theorem that the finitary binary corona of every uncountable locally finite group is a singleton \cite[p.~159]{Protasov2005}. This says that the full Higson corona is \emph{connected}; connectedness alone does not force it to be a singleton. For example, the interval $[0,1]$ has only two clopen sets, just as a one-point space does.

This expository article connects Holt's combinatorial argument to the $C^*$-algebraic Higson corona framework. Sections~\ref{sec:coarse_spaces} and~\ref{sec:higson_corona} pass from finite translation defects to coarse boundaries and then to projections in the Higson corona algebra. Section~\ref{sec:geometric_rigidity} gives a detailed account of Holt's coset argument in the cohomology-free form of Cornulier \cite[Section~10]{Cornulier}, illustrating the key step in the symmetric group on three letters. (Cornulier also extended Holt's theorem to non-$\sigma$-compact regionally elliptic locally compact groups \cite[Theorem~1.26]{Cornulier}.) Section~\ref{sec:main_theorem} derives the coarse-geometric form of Holt's one-end theorem and gives a self-contained proof of Protasov's result for uncountable abelian groups \cite[Proposition~4]{Protasov2005}, explaining why the full corona of $G_I$ is a singleton when $I$ is uncountable.

\medskip\noindent\textbf{Conventions.} All groups carry the discrete topology. For a subset $A$ of a set $X$, we write $\chi_A\colon X\to\{0,1\}$ for the characteristic function of $A$.

\section{Finite translation defects and coarse boundaries}\label{sec:coarse_spaces}

A coarse structure specifies which pairs of points are uniformly close at a large scale, without requiring a particular metric. Formally, a \emph{coarse structure} on a set $X$ is a family $\mathcal E$ of subsets of $X\times X$ containing the diagonal $\Delta_X \coloneqq \{(x,x):x\in X\}$ and closed under subsets, finite unions, inverses, and compositions, where
\[
  E^{-1} \coloneqq \{(y,x):(x,y)\in E\},
\]
and
\[
  E\circ F
  \coloneqq \{(x,z):\text{there exists }y\in X
    \text{ with }(x,y)\in E\text{ and }(y,z)\in F\}.
\]
The pair $(X,\mathcal E)$ is a \emph{coarse space}, and the members of $\mathcal E$ are called \emph{entourages} or \emph{controlled relations}; see \cite{Roe}.

The coarse space is \emph{connected} if every pair $(x,y)\in X\times X$ belongs to some entourage; equivalently,
\[
  \bigcup_{E\in\mathcal E}E=X\times X.
\]
The entourage may depend on the pair. This is connectedness of the coarse structure and does not refer to a topology on $X$.

We use the convention
\[
  E[A]\coloneqq\{x\in X:(x,a)\in E\text{ for some }a\in A\}.
\]
A subset $B\subseteq X$ is \emph{bounded} if $B\times B\in\mathcal E$. For nonempty $B$, this is equivalent to $B\subseteq E[\{a\}]$ for some entourage $E$ and some $a\in X$. It is \emph{cobounded} if $X\setminus B$ is bounded. A subset $A\subseteq X$ is \emph{coarsely clopen} if its outer boundary $E[A]\setminus A$ is bounded for every entourage $E$.

In a connected coarse space, finite unions of bounded sets are bounded. Indeed, for nonempty bounded sets $B\subseteq E[\{a\}]$ and $C\subseteq F[\{b\}]$, choose an entourage $E'$ containing $(b,a)$; then $B\cup C\subseteq(E\cup(F\circ E'))[\{a\}]$ (and the case where $B$ or $C$ is empty is immediate). Conversely, if finite unions of bounded sets are bounded, then every two-point set is bounded (since singletons are bounded via $\Delta_X\in\mathcal E$), so every pair belongs to an entourage.

The \emph{finitary coarse structure} on a group $G$ consists of all subsets of the relations
\[
  E_K \coloneqq \{(x,y)\in G\times G : x^{-1}y\in K\},\qquad K\subseteq G\text{ finite}.
\]
Thus a controlled pair has the form $(x,xk)$ for one of finitely many displacements $k$. This coarse structure is connected, since $(x,y)\in E_{\{x^{-1}y\}}$ for every $x,y\in G$. With our convention,
\[
  E_K[\{a\}]
    =\{x\in G:x^{-1}a\in K\}
    =aK^{-1}.
\]
Hence every bounded subset is finite. Conversely, a finite subset $B\subseteq G$ is contained in $E_{B^{-1}}[\{e\}]=B$, where $e$ is the identity of $G$. Thus the bounded subsets are precisely the finite subsets.

\begin{lemma}\label{lem:coarse_properties}
In the finitary coarse structure on $G$, a subset $A\subseteq G$ is coarsely clopen if and only if $Ag\triangle A$ is finite for every $g\in G$.
\end{lemma}

\begin{proof}
The equality $x^{-1}a=k$ gives $x=ak^{-1}$, so $E_K[A]=AK^{-1}$. Replacing $K$ by $K^{-1}$, we see that $A$ is coarsely clopen if and only if $AK\setminus A$ is finite for every finite $K\subseteq G$. This implies that both $Ag\setminus A$ and $A\setminus Ag=(Ag^{-1}\setminus A)g$ are finite for every $g\in G$. Conversely, $AK\setminus A$ is the finite union of the sets $Ag\setminus A$ over $g\in K$.
\end{proof}

\section{What the Higson corona detects}\label{sec:higson_corona}

Let $G$ be an infinite group. A function $f\colon G\to\mathbb{C}$ belongs to $C_0(G)$ if for every $\varepsilon>0$ the set $\{x\in G : |f(x)|\ge\varepsilon\}$ is finite. A bounded function $f\colon G\to\mathbb{C}$ is \emph{slowly oscillating}, or \emph{Higson}, if for every entourage $E$ and every $\varepsilon>0$ there is a finite subset $F\subseteq G$ such that $|f(x)-f(y)|<\varepsilon$ whenever $(x,y)\in E$ and $x,y\notin F$. In the finitary coarse structure this is equivalent to
\begin{equation}\label{eq:translate}
  \bigl(x\longmapsto f(xg)-f(x)\bigr)\in C_0(G)\qquad(g\in G).
\end{equation}
Indeed, for a fixed $g\in G$ the extra points excluded by $xg\in F$ form the finite set $Fg^{-1}$; conversely, a finite union of exceptional sets handles all $g\in K$ for an entourage $E_K$.

Write $C_h(G)$ for the unital algebra of Higson functions on $G$. Equipped with the supremum norm, $C_h(G)$ is a closed, self-adjoint unital subalgebra of the $C^*$-algebra of bounded functions $G\to\mathbb{C}$, and $C_0(G)$ is a closed ideal in $C_h(G)$. Since $G$ is infinite, $1\notin C_0(G)$, so the quotient $C^*$-algebra $C_h(G)/C_0(G)$ is nonzero. The \emph{Higson corona} $\partial G$ is the nonempty compact Hausdorff spectrum of this quotient; by commutative Gelfand duality,
\[
  C(\partial G)\cong C_h(G)/C_0(G).
\]
We write $[f]$ for the class of $f\in C_h(G)$ in $C_h(G)/C_0(G)$. A \emph{projection} in a $C^*$-algebra is an element $p$ satisfying $p=p^*=p^2$; in $C(K)$ for a compact Hausdorff space $K$, the projections are precisely the characteristic functions of clopen subsets of $K$.

Let $(X,\mathcal E)$ be an unbounded connected coarse space. For any entourage $E$ and any bounded subset $B\subseteq X$, the set $E[B]$ is bounded (trivially if $B=\varnothing$; otherwise, if $B\subseteq F[\{a\}]$, then $E[B]\subseteq(E\circ F)[\{a\}]$); in particular, every bounded subset of $X$ is coarsely clopen. Moreover, if $A$ and $B$ are coarsely clopen, then
\[
  E[X\setminus A]\setminus(X\setminus A)
  \subseteq E\bigl[E^{-1}[A]\setminus A\bigr]
  \quad\text{and}\quad
  E[A\cup B]\setminus(A\cup B)
  \subseteq (E[A]\setminus A)\cup(E[B]\setminus B)
\]
are bounded (using connectedness for the union), so the coarsely clopen subsets of $X$ form a Boolean algebra of sets containing all bounded subsets.

Write $\mathfrak{C}(X)$ for the quotient Boolean algebra of coarsely clopen subsets of $X$ modulo bounded symmetric difference. Thus $[A]=[B]$ if and only if $A\triangle B$ is bounded. Its operations and distinguished elements are
\[
  [A]\vee[B]=[A\cup B],\qquad
  [A]\wedge[B]=[A\cap B],\qquad
  [A]^c=[X\setminus A],
\]
\[
  0=[\varnothing],\qquad 1=[X].
\]
In particular, $[A]=0$ exactly when $A$ is bounded, and $[A]=1$ exactly when $A$ is cobounded. The order is given by $[A]\le[B]$ if and only if $A\setminus B$ is bounded. An \emph{ultrafilter} on this Boolean algebra is a family that contains $1$, excludes $0$, is closed under finite meets and passage to larger elements, and contains exactly one of $a$ and $a^c$ for every $a\in\mathfrak{C}(X)$.

Following Dydak, Ma, and Rashed \cite{MaRashedDydak}, a \emph{coarse end} of $X$ is a maximal family of unbounded coarsely clopen subsets of $X$ such that every finite subfamily has unbounded intersection. We denote a coarse end by $\mathfrak e$, reserving $\mathcal E$ for the coarse structure and $[A]$ for a class in $\mathfrak{C}(X)$. On a group with the finitary coarse structure, this agrees with the Stone-duality description of ends in \cite[Definition~1.5]{Cornulier} (with inversion interchanging left and right translation conventions).

\begin{proposition}[Coarse ends and coarsely clopen subsets]\label{prop:ends_clopen}
Let $(X,\mathcal E)$ be an unbounded connected coarse space. Then $X$ has exactly one coarse end if and only if every coarsely clopen subset of $X$ is bounded or cobounded.
\end{proposition}

\begin{proof}
We establish a bijection between the coarse ends of $X$ and the ultrafilters of $\mathfrak{C}(X)$. Let $\mathfrak e$ be a coarse end of $X$. By maximality, $\mathfrak e$ is closed under finite intersections and passage to coarsely clopen supersets. Membership in $\mathfrak e$ is invariant under bounded symmetric differences, because removing a bounded set from an unbounded set leaves an unbounded set. Set
\[
  \mathcal U_{\mathfrak e}\coloneqq\{[A]:A\in\mathfrak e\}.
\]
Since every member of $\mathfrak e$ is unbounded, $0=[\varnothing]\notin\mathcal U_{\mathfrak e}$. The family $\mathcal U_{\mathfrak e}$ contains $1=[X]$, is closed under finite meets, and is upward-closed.

Moreover, if $A$ is coarsely clopen and $A\notin\mathfrak e$, then maximality yields a finite subfamily of $\mathfrak e$ whose intersection $B\in\mathfrak e$ has bounded intersection with $A$. For any finite subfamily $C_1,\ldots,C_m\in\mathfrak e$, the intersection $C_1\cap\cdots\cap C_m\cap B$ is unbounded and meets $A$ in a bounded set, so $C_1\cap\cdots\cap C_m\cap B\cap(X\setminus A)$ is unbounded, and therefore $C_1\cap\cdots\cap C_m\cap(X\setminus A)$ is unbounded. By maximality, $X\setminus A\in\mathfrak e$. The sets $A$ and $X\setminus A$ cannot both belong to $\mathfrak e$, since their intersection is empty. Hence $\mathcal U_{\mathfrak e}$ is an ultrafilter on $\mathfrak{C}(X)$.

Conversely, given an ultrafilter $\mathcal{U}$ on $\mathfrak{C}(X)$, the family
\[
  \mathfrak e_{\mathcal U} \coloneqq \{A\subseteq X : A\text{ is coarsely clopen and }[A]\in\mathcal U\}
\]
consists of unbounded coarsely clopen subsets with unbounded finite intersections; it is maximal because any coarsely clopen $A\notin\mathfrak e_{\mathcal U}$ satisfies $[X\setminus A]\in\mathcal U$, so $A\cap(X\setminus A)=\varnothing$ is bounded. The assignments $\mathfrak e\mapsto\mathcal U_{\mathfrak e}$ and $\mathcal U\mapsto\mathfrak e_{\mathcal U}$ are mutually inverse.

If every coarsely clopen subset of $X$ is bounded or cobounded, then $\mathfrak{C}(X)$ is the two-element Boolean algebra, which has a unique ultrafilter. Conversely, if $A$ is coarsely clopen and neither bounded nor cobounded, then $[A]$ and $[X\setminus A]$ are both nonzero, so ultrafilters containing these two complementary classes give two distinct coarse ends of $X$.
\end{proof}

\begin{lemma}\label{lem:characteristic_higson}
For $A\subseteq G$, the function $\chi_A$ belongs to $C_h(G)$ if and only if $A$ is coarsely clopen.
\end{lemma}

\begin{proof}
For each $g\in G$, the function $x\mapsto\chi_A(xg)-\chi_A(x)$ is supported on $Ag^{-1}\triangle A$ and has absolute value $1$ there. Apply \eqref{eq:translate} and Lemma~\ref{lem:coarse_properties}.
\end{proof}

\begin{lemma}\label{lem:projections}
Every projection in $C(\partial G)$ has the form $[\chi_A]$ for some coarsely clopen $A\subseteq G$. This projection is nontrivial (i.e., distinct from $0$ and $1$) if and only if both $A$ and $G\setminus A$ are infinite.
\end{lemma}

\begin{proof}
Lift a projection $p\in C_h(G)/C_0(G)$ to $f\in C_h(G)$. Since $p=p^*$, the real part $q$ of $f$ also lifts $p$, and $q^2-q\in C_0(G)$. Set $A \coloneqq \{x\in G : q(x)\ge 1/2\}$. For every $t\in\mathbb{R}$,
\[
  \bigl|t-\chi_{[1/2,\infty)}(t)\bigr|\le 2|t(t-1)|,
\]
using $|t-1|\ge 1/2$ for $t<1/2$ and $|t|\ge 1/2$ for $t\ge 1/2$. Thus $q-\chi_A\in C_0(G)$, so $\chi_A\in C_h(G)$ and $p=[\chi_A]$. By Lemma~\ref{lem:characteristic_higson}, $A$ is coarsely clopen. Finally, $[\chi_A]=0$ if and only if $A$ is finite, and $[\chi_A]=1$ if and only if $G\setminus A$ is finite.
\end{proof}

Two characteristic functions define the same class in $C_h(G)/C_0(G)$ precisely when the corresponding subsets differ at finitely many points, and products and complements correspond to intersections and complements of sets. Thus $[A]\mapsto[\chi_A]$ identifies $\mathfrak{C}(G)$ with the Boolean algebra of projections in $C(\partial G)$.

\begin{corollary}\label{cor:corona_connected}
For an infinite group $G$ with the finitary coarse structure, the following are equivalent:
\begin{enumerate}
  \item $G$ has exactly one coarse end;
  \item every coarsely clopen subset of $G$ is finite or cofinite;
  \item the Higson corona $\partial G$ is connected.
\end{enumerate}
\end{corollary}

\begin{proof}
Proposition~\ref{prop:ends_clopen} gives the equivalence of (1) and (2). Since $G$ is infinite, $\partial G$ is a nonempty compact Hausdorff space. It is connected if and only if $C(\partial G)$ contains no nontrivial projections. Lemma~\ref{lem:projections} therefore gives the equivalence of (2) and (3).
\end{proof}

\begin{example}\label{ex:countable}
The opening coloring gives a nontrivial projection for $G_{\mathbb{N}}$, but parity is only one choice. For every $S\subseteq\mathbb{N}_0 \coloneqq \{0,1,\ldots\}$, the set $m^{-1}(S)$ has finite translation defects by the same calculation. Each fiber of $m$ is nonempty and finite: there is one set of maximum $0$ by convention, and $2^{n-1}$ sets of maximum $n\ge 1$. Thus $m^{-1}(S)\triangle m^{-1}(T)$ is finite if and only if $S\triangle T$ is finite. Taking inverse images therefore embeds $\mathcal{P}(\mathbb{N}_0)/\mathrm{Fin}$, the power set modulo finite sets, into the projection algebra of $C(\partial G_{\mathbb{N}})$.
\end{example}

\section{Why one color occurs only finitely often}\label{sec:geometric_rigidity}

Uncountability leaves room outside a countable subgroup that contains all of its own translation defects.

\begin{lemma}[Countable defect hull]\label{lem:defect_hull}
Let $G$ be an uncountable group, and suppose that a finite subset $D_g\subseteq G$ is assigned to every $g\in G$. Given a countably infinite subset $C\subseteq G$, there exists a countably infinite proper subgroup $H<G$ containing $C$ such that
\[
  D_h\subseteq H \qquad\text{for every }h\in H.
\]
\end{lemma}

\begin{proof}
Set $H_0 \coloneqq \langle C\rangle$. Having constructed a countable subgroup $H_n$, define
\[
  H_{n+1} \coloneqq \left\langle H_n\cup\bigcup_{h\in H_n}D_h\right\rangle.
\]
The union inside the brackets is countable because $H_n$ is countable and every $D_h$ is finite, so $H_{n+1}$ is countable. Because $H_0\le H_1\le\cdots$ is an increasing chain of subgroups,
\[
  H \coloneqq \bigcup_{n=0}^{\infty}H_n
\]
is a countably infinite subgroup containing $C$; it is proper because $G$ is uncountable. If $h\in H$, then $h\in H_n$ for some $n$, so $D_h\subseteq H_{n+1}\subseteq H$.
\end{proof}

The following lemma is due to Holt \cite[\S3]{Holt1981}.
We present his alternating-coset argument in terms of a
bipartite coset graph, making the connecting path explicit.
Cornulier \cite[Lemma~10.1]{Cornulier} gives a more general
formulation.

\begin{lemma}[Holt's coset-connectivity lemma]
\label{lem:holt_cosets}
Let $\Gamma$ be a periodic group (i.e., every element has finite order) generated by two proper
subgroups $H$ and $K$. Let $S$ be a set, and suppose that
$f\colon\Gamma\to S$ is constant on each right coset of $H$
other than $H$ itself, and on each right coset of $K$
other than $K$ itself.
Then $f$ is constant on $\Gamma\setminus(H\cap K)$.
If $f$ is also constant on at least one of the subgroups
$H$ and $K$, then it is constant on $\Gamma$.
\end{lemma}

\begin{proof}
Consider the bipartite coset graph $\mathcal{X}$ whose
vertex set is the disjoint union
\[
  V(\mathcal{X})
  \coloneqq H\backslash\Gamma \;\sqcup\; K\backslash\Gamma,
\]
where $H\backslash\Gamma=\{Hg:g\in\Gamma\}$ and
$K\backslash\Gamma=\{Kg:g\in\Gamma\}$ denote the sets of right
cosets, and whose edge set is $(H\cap K)\backslash\Gamma$, with
the edge $(H\cap K)g$ joining $Hg$ to $Kg$. For any
$g_1,g_2\in\Gamma$, we have $(H\cap K)g_1=(H\cap K)g_2$ if and
only if $g_1g_2^{-1}\in H\cap K$, which holds if and only if both
$Hg_1=Hg_2$ and $Kg_1=Kg_2$. Thus incidence is well-defined and
$\mathcal{X}$ is a simple bipartite graph; moreover, every vertex
$Hg$ or $Kg$ is incident to $(H\cap K)g$, so $\mathcal{X}$ has no
isolated vertices.

We first show that $\mathcal{X}$ is connected. Right
multiplication by $\Gamma$ acts on $\mathcal{X}$ by graph
automorphisms and is transitive on the edge set
$(H\cap K)\backslash\Gamma$. Let $\mathcal{Y}$ be the connected
component containing the base edge $H\cap K$, whose endpoints are
$H$ and $K$. For every $h\in H$, right multiplication by $h$ fixes
$H$ and hence preserves $\mathcal{Y}$; similarly, every element of
$K$ fixes $K$ and preserves $\mathcal{Y}$. Since
$\Gamma=\langle H,K\rangle$, the component $\mathcal{Y}$ is
$\Gamma$-invariant. Edge transitivity implies that $\mathcal{Y}$
contains every edge of $\mathcal{X}$, and since $\mathcal{X}$ has
no isolated vertices, $\mathcal{Y}=\mathcal{X}$.

Let $\mathcal{X}^{\circ}$ be the induced subgraph obtained by
deleting the vertices $H$ and $K$ together with all edges incident
to them. We claim that $\mathcal{X}^{\circ}$ is connected.
Because $\mathcal{X}$ is connected, every connected component of
$\mathcal{X}^{\circ}$ contains at least one vertex adjacent in
$\mathcal{X}$ to $H$ or $K$. The edges of $\mathcal{X}$ incident
to $H$ are precisely $(H\cap K)h$ ($h\in H$), whose opposite
endpoints are $Kh$; such a vertex belongs to $\mathcal{X}^{\circ}$
if and only if $Kh\neq K$, i.e., $h\in H\setminus K$. Likewise,
the neighbors of $K$ in $\mathcal{X}^{\circ}$ are precisely the
vertices $Hk$ with $k\in K\setminus H$. Since $H$ and $K$ are
proper subgroups generating $\Gamma$, neither contains the other,
so both neighbor sets
\[
  N_H^{\circ} \coloneqq \{Kh : h\in H\setminus K\}
  \qquad\text{and}\qquad
  N_K^{\circ} \coloneqq \{Hk : k\in K\setminus H\}
\]
are nonempty. Consequently, to prove that $\mathcal{X}^{\circ}$ is
connected, it suffices to join an arbitrary $Hk\in N_K^{\circ}$ to
an arbitrary $Kh\in N_H^{\circ}$ by a path in
$\mathcal{X}^{\circ}$ (nonemptiness of both sets then forces
$N_H^{\circ}\cup N_K^{\circ}$ to lie in a single connected
component of $\mathcal{X}^{\circ}$).

Fix $h\in H\setminus K$ and $k\in K\setminus H$, and set
\[
  s \coloneqq h^{-1}k
  \qquad\text{and}\qquad
  C \coloneqq \langle s\rangle.
\]
Since $\Gamma$ is periodic, $C$ is finite cyclic. Let $m$ and $n$
be the cardinalities of the right $C$-orbits of the vertices $H$
and $K$, respectively. Because $C=\langle s\rangle$ is finite
cyclic, $m$ is the least positive integer such that $Hs^m=H$ (so
the cosets $H,Hs,\ldots,Hs^{m-1}$ are pairwise distinct), and $n$
is the least positive integer such that $Ks^n=K$ (so the cosets
$K,Ks,\ldots,Ks^{n-1}$ are pairwise distinct). If $m<n$, we
interchange $(H,h)$ and $(K,k)$; this replaces $s$ by
$k^{-1}h=s^{-1}$, leaves $C=\langle s\rangle=\langle s^{-1}\rangle$
unchanged, swaps $m$ and $n$, and merely reverses the two vertices
to be connected. Thus we may assume $m\ge n$. Moreover,
$s\notin K$, since $s\in K$ would give $h^{-1}=sk^{-1}\in K$,
contradicting $h\notin K$. Hence $Ks\neq K$ and therefore
$n\ge 2$.

Consider the sequence of vertices
\[
  Hs,\ Ks,\ Hs^2,\ Ks^2,\ \ldots,\ Hs^{n-1},\ Ks^{n-1}
\]
(see Fig.~\ref{fig:holt_coset_path}). For every
$1\le i\le n-1$, we have $Ks^i\neq K$ by the minimality of $n$,
and $Hs^i\neq H$ by the minimality of $m$ together with $m\ge n$.
The minimality of $m$ and $n$ also makes the displayed vertices
distinct within each coset family; vertices from the two families
are distinct because the vertex set is a disjoint union.
Thus all displayed vertices are distinct vertices of
$\mathcal{X}^{\circ}$. Consecutive vertices in this
sequence are joined by edges of $\mathcal{X}^{\circ}$:
\begin{itemize}
  \item for $1\le i\le n-1$, the edge $(H\cap K)s^i$ joins
  $Hs^i$ to $Ks^i$;
  \item for $1\le i\le n-2$ (vacuously when $n=2$), the edge
  $(H\cap K)ks^i$ joins $Kks^i=Ks^i$ to $Hks^i=Hhs^{i+1}=Hs^{i+1}$,
  using $k\in K$, $h\in H$, and $k=hs$.
\end{itemize}
Finally, the endpoints of this path are
\[
  Hs = Hh^{-1}k = Hk
  \qquad\text{and}\qquad
  Ks^{n-1} = Ks^{-1} = Kk^{-1}h = Kh,
\]
where the second chain of equalities uses $Ks^n=K$ and
$k^{-1}\in K$. Hence $Hk$ and $Kh$ are connected in
$\mathcal{X}^{\circ}$, and therefore $\mathcal{X}^{\circ}$ is
connected.

By hypothesis, $f$ is constant on each vertex (right coset) of
$\mathcal{X}^{\circ}$, so it induces a well-defined map
$\bar{f}\colon V(\mathcal{X}^{\circ})\to S$ satisfying
$\bar{f}(Hg)=f(g)$ whenever $Hg\neq H$ and $\bar{f}(Kg)=f(g)$
whenever $Kg\neq K$. For every edge $(H\cap K)g$ of
$\mathcal{X}^{\circ}$, both endpoints $Hg$ and $Kg$ belong to
$\mathcal{X}^{\circ}$ and contain $g$, so
$\bar{f}(Hg)=f(g)=\bar{f}(Kg)$. Since $\mathcal{X}^{\circ}$ is
connected, $\bar{f}$ takes a single constant value $c\in S$ on
all of $V(\mathcal{X}^{\circ})$.

If $g\in\Gamma\setminus(H\cap K)$, then either $g\notin H$ (so
$Hg\neq H$ lies in $\mathcal{X}^{\circ}$) or $g\notin K$ (so
$Kg\neq K$ lies in $\mathcal{X}^{\circ}$), and in either case
$f(g)=c$. Thus $f$ is constant on $\Gamma\setminus(H\cap K)$.
For the final assertion, symmetry between $H$ and $K$ allows us to
assume that $f$ is constant on $K$. Choosing $k_0\in K\setminus H$
(which is nonempty as established above), we have
$k_0\in\Gamma\setminus(H\cap K)$, so the common value of $f$ on
$K$ is $f(k_0)=c$. In particular, $f=c$ on $H\cap K\subseteq K$,
and hence on all of $\Gamma$.
\end{proof}

\begin{figure}[htbp]
\centering
\begin{tikzpicture}[
  basev/.style={circle, draw=black!50, fill=white, thick, inner sep=2.2pt},
  actv/.style={circle, draw=blue!70!black, fill=blue!70!black, inner sep=2.2pt},
  deledge/.style={dashed, black!45, thick},
  pathedge/.style={solid, blue!70!black, very thick},
  edgelbl/.style={font=\scriptsize, text=black!85}
]
  \node[basev, label={[black!65]left:$H$}]             (H)   at (0,   2.6) {};
  \node[actv,  label=above:{$Hs = Hk$}]                (Hs1) at (2.9, 2.6) {};
  \node[actv,  label=above:{$Hs^2$}]                   (Hs2) at (5.8, 2.6) {};
  \node[actv,  label=above:{$Hs^{n-1}$}]               (Hsn) at (9.4, 2.6) {};

  \node[basev, label={[black!65]left:$K$}]             (K)   at (0,   0)   {};
  \node[actv,  label=below:{$Ks$}]                     (Ks1) at (2.9, 0)   {};
  \node[actv,  label=below:{$Ks^2$}]                   (Ks2) at (5.8, 0)   {};
  \node[actv,  label=below:{$Ks^{n-1} = Kh$}]          (Ksn) at (9.4, 0)   {};

  \draw[deledge] (H) -- node[edgelbl, text=black!60, fill=white, inner sep=1.5pt] {$H\cap K$} (K);
  \draw[deledge] (K) -- node[edgelbl, text=black!60, sloped, above, pos=0.52] {$(H\cap K)k$} (Hs1);
  \draw[deledge, rounded corners=8pt]
    (H) -- (0, 3.55)
        -- node[edgelbl, text=black!60, above] {$(H\cap K)h$} (10.6, 3.55)
        -- (10.6, 0) -- (Ksn);

  \draw[pathedge] (Hs1) -- node[edgelbl, fill=white, inner sep=1.5pt] {$(H\cap K)s$} (Ks1);
  \draw[pathedge] (Ks1) -- node[edgelbl, sloped, above, pos=0.55] {$(H\cap K)ks$} (Hs2);
  \draw[pathedge] (Hs2) -- node[edgelbl, fill=white, inner sep=1.5pt] {$(H\cap K)s^2$} (Ks2);

  \draw[pathedge] (Ks2) -- (6.7, 0.65);
  \node[font=\large, text=blue!70!black] at (7.6, 1.3) {$\cdots$};
  \draw[pathedge] (8.5, 1.95) -- (Hsn);

  \draw[pathedge] (Hsn) -- node[edgelbl, fill=white, inner sep=1.5pt] {$(H\cap K)s^{n-1}$} (Ksn);
\end{tikzpicture}
\caption{Schematic of the connecting path in
Lemma~\ref{lem:holt_cosets}, drawn for $m\ge n\ge4$ with
the intermediate part abbreviated. Only the indicated
vertices and edges are shown. Open circles mark the deleted
vertices $H$ and $K$, and dashed edges meet these vertices.
The solid path in $\mathcal{X}^\circ$ connects $Hk=Hs$ to
$Kh=Ks^{n-1}$. For $n=2$ the path is a single edge; for
$n=3$ it has three edges.}
\label{fig:holt_coset_path}
\end{figure}

\begin{example}\label{ex:hexagon}
Take $\Gamma=S_3$, $h=(12)$, $k=(23)$, $H=\langle h\rangle$, and $K=\langle k\rangle$. Here $s=h^{-1}k$ has order three. Each subgroup has three right cosets, their intersection is trivial, and the six edges form a hexagon. Once the vertices $H$ and $K$ are removed, the path from $Hk=Hs$ to $Kh=Ks^2$ still connects all four surviving vertices (Fig.~\ref{fig:hexagon}). The argument in Lemma~\ref{lem:holt_cosets} constructs such a detour without assuming that the two coset orbits have equal lengths.
\end{example}

\begin{figure}[htbp]
\centering
\begin{tikzpicture}[
  basev/.style={circle, draw=black!50, fill=white, thick, inner sep=2.2pt},
  actv/.style={circle, draw=blue!70!black, fill=blue!70!black, inner sep=2.2pt},
  deledge/.style={dashed, black!45, thick},
  pathedge/.style={solid, blue!70!black, very thick}
]
  \node[basev, label={[black!65]left:$H$}]   (H)   at (-2.2,  0)   {};
  \node[basev, label={[black!65]above:$K$}]  (K)   at (-1.1,  1.2) {};
  \node[actv,  label=above:{$Hs$}]           (Hs)  at ( 1.1,  1.2) {};
  \node[actv,  label=right:{$Ks$}]           (Ks)  at ( 2.2,  0)   {};
  \node[actv,  label=below:{$Hs^2$}]         (Hs2) at ( 1.1, -1.2) {};
  \node[actv,  label=below:{$Ks^2$}]         (Ks2) at (-1.1, -1.2) {};

  \draw[deledge]  (H)   -- (K);
  \draw[deledge]  (K)   -- (Hs);
  \draw[deledge]  (H)   -- (Ks2);
  \draw[pathedge] (Hs)  -- (Ks);
  \draw[pathedge] (Ks)  -- (Hs2);
  \draw[pathedge] (Hs2) -- (Ks2);
\end{tikzpicture}
\caption{The coset graph for $S_3$. Open circles mark the deleted vertices $H$ and $K$, dashed segments mark their incident edges, and the three solid edges form the remaining path in $\mathcal{X}^{\circ}$.}
\label{fig:hexagon}
\end{figure}

The proof of Holt's theorem below follows the discrete case of \cite[Section~10]{Cornulier}.

\begin{theorem}[Holt's rigidity theorem {\cite{Holt1981}}]\label{thm:holt_rigidity}
Let $G$ be an uncountable locally finite group. If $A\subseteq G$ and $Ag\triangle A$ is finite for every $g\in G$, then $A$ is finite or cofinite.
\end{theorem}

\begin{proof}
Constancy on right cosets expresses invariance under left translation, whereas our hypothesis controls right-translation defects. To pass between them, set $B \coloneqq A^{-1}$ and $\alpha \coloneqq \chi_B$. For each $g\in G$, the exceptional set
\[
  D_g \coloneqq \{x\in G : \alpha(gx)\neq\alpha(x)\}
      = g^{-1}B\triangle B = (Ag\triangle A)^{-1}
\]
is finite.

Choose a countably infinite subset of $G$ and apply Lemma~\ref{lem:defect_hull} to it and to $(D_g)_{g\in G}$. This produces a countably infinite proper subgroup $H<G$ containing $D_h$ for every $h\in H$. If $x\notin H$, then $x\notin D_h$, so $\alpha(hx)=\alpha(x)$ for every $h\in H$. Thus $\alpha$ is constant on every right $H$-coset other than $H$.

Fix $k\in G\setminus H$ and put $K \coloneqq \langle k\rangle$ and $\Gamma_k \coloneqq \langle H,K\rangle$. Local finiteness makes $K$ finite.

Call a right $K$-coset $Kx\subseteq\Gamma_k$ \emph{exceptional} if $\alpha$ is not constant on it, and let $\mathcal{C}_K$ denote the family of exceptional right $K$-cosets in $\Gamma_k$. If $\alpha(qx)\neq\alpha(rx)$ for $q,r\in K$, then $rx\in Kx\cap D_{qr^{-1}}$. Hence every coset in $\mathcal{C}_K$ meets the finite set $F_K \coloneqq \bigcup_{z\in K}D_z$; since right $K$-cosets are pairwise disjoint, $\mathcal{C}_K$ is finite.

For each coset $C\in\mathcal{C}_K$, choose a representative $r_C\in C\cap F_K$, and express these finitely many representatives as finite words in elements of $H\cup K$. Let $H_k^0$ be the subgroup of $H$ generated by all the $H$-letters occurring in those words (with $H_k^0\coloneqq\{e\}$ if no $H$-letters occur, in particular if $\mathcal{C}_K=\varnothing$). By local finiteness, both $H_k^0$ and $L_k \coloneqq \langle H_k^0,K\rangle$ are finite. Moreover, $L_k$ contains every chosen representative; since $K\le L_k$, each exceptional coset $C=Kr_C$ is contained in $L_k$. Consequently,
\begin{equation}\label{eq:star}
  \alpha\text{ is constant on }Kx
  \qquad\text{whenever }x\in\Gamma_k\setminus L_k.
\end{equation}

We claim that $\alpha$ is constant on every right $L_k$-coset contained in $\Gamma_k$ and distinct from $L_k$. Fix such a coset $W\subseteq\Gamma_k$. Choose a representative $g\in W$ with $g\in W\cap H$ whenever $W\cap H\neq\varnothing$; then $W=L_k g$ with $g\in\Gamma_k\setminus L_k$, and
\begin{equation}\label{eq:dstar}
  lg\in H\text{ for some }l\in L_k
  \quad\Longrightarrow\quad
  g\in H.
\end{equation}
Define
\[
  \alpha_g\colon L_k\to\{0,1\},
  \qquad
  \alpha_g(l) \coloneqq \alpha(lg).
\]

If $L_k=K$, then $\alpha$ is constant on $L_k g=Kg$ by \eqref{eq:star}, so the claim holds. Assume henceforth that $K<L_k$ and set $P \coloneqq H\cap L_k$. Because $H_k^0\le P$, we have $L_k=\langle P,K\rangle$. Both $P$ and $K$ are proper subgroups of $L_k$: $P$ is proper because $k\in L_k\setminus H$, and $K$ is proper by assumption.

For every $l\in L_k$, we have $lg\notin L_k$ because $g\notin L_k$. By \eqref{eq:star}, $\alpha$ is constant on $Klg$, so $\alpha_g$ is constant on every right $K$-coset $Kl$ in $L_k$, including $K$ itself.

Now let $Pl\subseteq L_k$ be a right $P$-coset distinct from $P$, with $l\in L_k$. If $lg\in H$, then \eqref{eq:dstar} gives $g\in H$, hence $l=(lg)g^{-1}\in H\cap L_k=P$, a contradiction. Thus $lg\notin H$. Since $\alpha$ is constant on $Hlg$ and $Plg\subseteq Hlg$, the function $\alpha_g$ is constant on $Pl$.

Applying the second assertion of Lemma~\ref{lem:holt_cosets} inside the finite group $L_k$ to the subgroups $P$ and $K$, we conclude that $\alpha_g$ is constant on all of $L_k$; equivalently, $\alpha$ is constant on $L_k g=W$. This proves the claim.

Now apply Lemma~\ref{lem:holt_cosets} a second time inside $\Gamma_k=\langle H,L_k\rangle$ (where the equality holds because $K\le L_k\le\Gamma_k$), which is periodic because $G$ is locally finite. Both subgroups are proper: $k\notin H$, whereas $L_k$ is finite and $H$ is infinite. Since $\alpha$ is constant on every right coset of $H$ other than $H$ and on every right coset of $L_k$ other than $L_k$, Lemma~\ref{lem:holt_cosets} yields a single value $c_k\in\{0,1\}$ on $\Gamma_k\setminus(H\cap L_k)$. Since $k\notin H$, we have $\alpha(k)=c_k$; also $\alpha(h)=c_k$ for every $h\in H\setminus L_k$.

The value $c_k$ is independent of $k$. Indeed, if $k,k'\in G\setminus H$, then $L_k\cup L_{k'}$ is finite. Since $H$ is infinite, we may choose $h\in H\setminus(L_k\cup L_{k'})$, giving $c_k=\alpha(h)=c_{k'}$. Denote this common value by $c$. Then $\alpha=c$ on $G\setminus H$.

Fix $k_0\in G\setminus H$. We also have $\alpha=c$ on $H\setminus L_{k_0}$, so $\alpha=c$ outside the finite set $H\cap L_{k_0}$. Therefore $B$ is finite if $c=0$ and cofinite if $c=1$. Since inversion is a bijection of $G$, the same conclusion holds for $A=B^{-1}$.
\end{proof}

\section{One end and the size of the corona}\label{sec:main_theorem}

\begin{theorem}[Holt's one-end theorem in coarse form]\label{thm:main_one_end}
Let $G$ be an uncountable locally finite group equipped with the finitary coarse structure. Then every coarsely clopen subset of $G$ is finite or cofinite. Consequently, the Higson corona $\partial G$ is connected and $G$ has exactly one coarse end.
\end{theorem}

\begin{proof}
Let $A\subseteq G$ be coarsely clopen. By Lemma~\ref{lem:coarse_properties}, $Ag\triangle A$ is finite for every $g\in G$. Theorem~\ref{thm:holt_rigidity} therefore implies that $A$ is finite or cofinite. Since $G$ is infinite, Corollary~\ref{cor:corona_connected} shows that $\partial G$ is connected and that $G$ has exactly one coarse end.
\end{proof}

The conclusion about ends is the finitary binary-corona assertion recorded by Protasov \cite[p.~159]{Protasov2005}. For $G_I$ with $I$ uncountable, it rules out the coloring from Section~\ref{sec:intro}. To conclude that \emph{every} Higson function on $G_I$ is constant modulo $C_0(G_I)$---equivalently, that the full Higson corona $\partial G_I$ is a singleton---one needs more than the argument for two colors. The following result of Protasov shows where commutativity enters.

\begin{proposition}[Uncountable abelian groups {\cite[Proposition~4, with $\kappa=\aleph_0$]{Protasov2005}}]\label{prop:abelian}
Let $G$ be an uncountable abelian group equipped with the finitary coarse structure. Then the Higson corona $\partial G$ is a singleton.
\end{proposition}

\begin{proof}
Write $G$ additively and fix $f\in C_h(G)$. Suppose for contradiction that $f$ is not constant modulo $C_0(G)$. We claim that there exists $\varepsilon>0$ such that outside every finite subset of $G$ there are elements $u,v$ with $|f(u)-f(v)|\ge\varepsilon$. Indeed, otherwise, for each $m\ge 1$ there would exist a finite subset $F_m\subseteq G$ such that $|f(u)-f(v)|<1/m$ for all $u,v\in G\setminus F_m$; choosing $x_m\in G\setminus(F_1\cup\cdots\cup F_m)$ would yield a Cauchy sequence $(f(x_m))$ in $\mathbb{C}$ whose limit $c\in\mathbb{C}$ satisfies $|f(x)-c|\le 1/m$ for all $x\in G\setminus F_m$, whence $f-c\in C_0(G)$, a contradiction.

Recursively choose pairs $(u_n,v_n)$ ($n\ge 1$) with $|f(u_n)-f(v_n)|\ge\varepsilon$ by selecting each pair outside $\{u_1,v_1,\ldots,u_{n-1},v_{n-1}\}$. Since $|f(u_n)-f(v_n)|\ge\varepsilon>0$, we automatically have $u_n\neq v_n$, so all elements $u_1,v_1,u_2,v_2,\ldots$ are pairwise distinct. Set $S \coloneqq \{u_n,v_n : n\ge 1\}$.

For each $s\in S$ and each integer $j\ge 1$, the set
\[
  D_{s,j} \coloneqq \{x\in G : |f(x+s)-f(x)|\ge 1/j\}
\]
is finite by \eqref{eq:translate}. Since $S$ is countable and $G$ is uncountable, the union $\bigcup_{s\in S,\, j\ge 1} D_{s,j}$ is countable, so we may choose $t\in G$ outside it. Then $f(t+s)=f(t)$ for every $s\in S$; in particular, $f(t+u_n)=f(t+v_n)=f(t)$ for all $n\ge 1$. Applying \eqref{eq:translate} to the single fixed translation $t$, and noting that, since the elements $u_1,v_1,u_2,v_2,\ldots$ are pairwise distinct, both sequences $(u_n)$ and $(v_n)$ eventually leave every finite subset of $G$, we obtain $f(u_n+t)-f(u_n)\to 0$ and $f(v_n+t)-f(v_n)\to 0$ as $n\to\infty$. By commutativity, $f(u_n+t)=f(t+u_n)=f(t)$ and $f(v_n+t)=f(t+v_n)=f(t)$, so both $f(u_n)$ and $f(v_n)$ converge to $f(t)$, contradicting $|f(u_n)-f(v_n)|\ge\varepsilon$.

Hence every $f\in C_h(G)$ is constant modulo $C_0(G)$, so $C(\partial G)\cong\mathbb{C}$ and $\partial G$ is a singleton.
\end{proof}

In particular, for any uncountable set $I$, the full Higson
corona of $G_I=\bigoplus_{i\in I}\mathbb{Z}/2\mathbb{Z}$ is a
singleton. Local finiteness plays no role in
Proposition~\ref{prop:abelian}. Theorem~\ref{thm:main_one_end}
applies without commutativity and establishes connectedness
of the full corona without asserting that it is a singleton.
For a general $f\in C_h(G)$ and each fixed $g\in G$, the set
\[
  \{x\in G:f(xg)\neq f(x)\}
  =\bigcup_{j\ge1}
    \{x\in G:|f(xg)-f(x)|\ge 1/j\}
\]
is at most countable. Applying inversion gives the analogous
statement for the left-translation defects of
$x\mapsto f(x^{-1})$. Local finiteness does not ensure that
a subgroup generated by countably many exceptional elements
is finite, so the finite-subgroup step in the proof of
Theorem~\ref{thm:holt_rigidity} does not extend directly to
arbitrary Higson functions. The arguments given here therefore
do not determine the full corona for a general uncountable
locally finite group.

\section*{Acknowledgments}
The author thanks Jerzy Dydak for discussions about the coarse geometry of groups that helped shape this exposition.

\section*{Statements and Declarations}
\noindent\textbf{Funding.} \FundingDisclosure\par\smallskip
\noindent\textbf{Competing Interests.} \InterestDisclosure\par\smallskip
\noindent\textbf{Data availability.} This is a theoretical article. No datasets were generated or analyzed.\par

\end{document}